\documentclass[11pt]{amsart}
\usepackage[matrix,arrow,tips,curve]{xy}

\usepackage{color}
\numberwithin{equation}{section}

\makeatletter
 \renewcommand\section{\@startsection {section}{1}{\z@}%
     {-4.5ex \@plus -1ex \@minus -.2ex}%
     {2.3ex \@plus.8ex}%
    {\centering\scshape}}

\def\C{\mathcal{C}}

\def\mcO{\mathcal{O}}

\def\T{\mathcal{T}}

\def\W{\mathcal{W}}

\def\PP{\mathbb{P}}

\def\ov#1{\overline{#1}}

\def\lra{\longrightarrow}

\def\End{\operatorname{End}}

\def\coker{\operatorname{Coker}}
\def\ker{\operatorname{Ker}}
\def\im{\operatorname{Im}}

\def\oc2{\mathcal{O}_{\C_2}}

\newcommand{\dblq}{{/\!/}}

\newtheorem{theorem}{Theorem}[section]
\newtheorem{prop}[theorem]{Proposition}

\newtheorem{lem}[theorem]{Lemma}
\theoremstyle{definition}
\newtheorem{defin}[theorem]{Definition}

\newtheorem{rem}[theorem]{Remark}
\newtheorem{cor}[theorem]{Corollary}

\newcommand{\be}{\begin{equation}}
\newcommand{\ee}{\end{equation}}

\begin{document}

\title {Explicit Brill-Noether-Petri general curves}

\author{Enrico Arbarello}
\address{Enrico Arbarello: Dipartimento di Matematica Guido Castelnuovo, Universit\`a di Roma Sapienza
\hfill
\indent Piazzale A. Moro 2, 00185 Roma, Italy} \email{{\tt ea@mat.uniroma1.it}}
\author{Andrea Bruno}
\address{Andrea Bruno: Dipartimento di Matematica e Fisica, Universit\`a Roma Tre
\hfill \newline\texttt{}  \indent Largo San Leonardo Murialdo 1-00146 Roma, Italy} \hfill \newline\texttt{}
 \email{{\tt bruno@mat.uniroma3.it}}

 \author{Gavril Farkas}
\address{Gavril Farkas: Institut f\"ur Mathematik, Humboldt-Universit\"at zu Berlin \hfill \newline\texttt{}
\indent Unter den Linden 6,
10099 Berlin, Germany}
\email{{\tt farkas@math.hu-berlin.de}}
 \author{Giulia Sacc\`a}
\address{Giulia Sacc\`a: Department of Mathematics, Stony Brook University \hfill
\indent \newline\texttt{}
\indent Stony Brook, NY 11794-3651, USA}
 \email{{\tt giulia.sacca@stonybrook.edu    }}

\begin{abstract}Let $p_1,\dots, p_9$ be the
points in $\mathbb A^2(\mathbb Q)\subset \mathbb P^2(\mathbb Q)$ with coordinates
$$(-2,3),(-1,-4),(2,5),(4,9),(52,375), (5234, 37866),(8, -23), (43, 282), \Bigl(\frac{1}{4}, -\frac{33}{8} \Bigr)$$ respectively.
 We  prove that, for any genus $g$, a plane curve of degree $3g$ having   a $g$-tuple point at $p_1,\dots, p_8$, and a $(g-1)$-tuple
point at $p_9$, and no other singularities, exists and is a Brill-Noether general curve of genus $g$, while a general  curve in that $g$-dimensional linear system
is a Brill-Noether-Petri general curve of genus $g$.
\end{abstract}
%

\maketitle
\section{Introduction.}

The Petri Theorem asserts that for a general curve $C$ of genus $g>1$, the multiplication map
$$\mu_{0,L}:H^0(C,L)\otimes H^0(C, \omega_C\otimes L^{-1})\rightarrow H^0(C,\omega_C)$$
is injective for every line bundle $L$ on $C$. While the result, which immediately implies the Brill-Noether Theorem,  holds for almost every curve $[C]\in \mathcal{M}_g$, so far no explicitly computable examples of smooth curves of arbitrary genus satisfying this theorem have been known. Indeed, there are two types of known proofs
of the Petri Theorem. These are: the proofs by degeneration due to  Griffiths-Harris \cite{GH}, Gieseker \cite{Gieseker}, and Eisenbud-Harris \cite{eis-harris}, or the recent proof using tropical geometry \cite{Payne}, which by their nature, shed little light on the explicit smooth curves which are Petri general; and the elegant proof by Lazarsfeld \cite{Lazarsfeld}, asserting that every hyperplane section of a polarised $K3$ surface $(X,H)$ of degree $2g-2$, such that the hyperplane class $[H]$ is indecomposable is a Brill-Noether general curve, while a general curve in the linear system $|H|$ is Petri general. However, there are no known concrete examples of polarised $K3$ surfaces of arbitrary degree  satisfying the requirement above.   It is a non-trivial instance of a theorem of Andr\'e \cite{andre}, \cite{MP}, that there exists polarised $K3$ surfaces of degree $2g-2$ over a number field, having  Picard number one. While the above mentioned results are all in characteristic zero, it has been observed by Welters \cite{Welters} that a minor modification of the proof in \cite{eis-harris}, proves the Petri Theorem in positive characteristic as well.

\vskip 3pt

This work originated from the paper \cite{ABS2}, where a number of explicit families of curves lying on the projective plane or on a ruled  elliptic surface were constructed. For these curves the question of whether they satisfy
the Brill-Noether-Petri condition arises naturally. Among these families one,  already studied by du Val \cite{DuVal}, is particularly interesting. Curves in this family naturally sit on the blow-up of the projective plane in nine points.

\vskip 3pt

The aim of this paper is to show that, by using the methods from \cite{Lazarsfeld} and \cite{Pareschi}, coupled with  Nagata's classical results \cite{Nagata2} on the effective cone of the blown-up projective plane, these curves provide \emph{explicit} examples of Brill-Noether-Petri general curves of any genus. They also provide \emph{computable} examples of Brill-Noether general curves of any genus.

In \cite{Treibich}, Treibich sketches a construction of
Brill-Noether (but not necessarily Petri) curves of any given genus.

We  set the notation we are going to use throughout this note.
We denote by $S'$ the blow-up of $ \PP^2$ at nine points
 $p_1,\dots, p_{9}$ which are {\it $3g$-general} (see the Definition \ref{halph} below),
and we let $E_1,\dots, E_{9}$ be the exceptional curves of this blow-up. We have that
$$
-K_{S'}\sim 3\ell-E_1-\cdots-E_{9}\,,
$$
where $\ell$  is the proper transform of a line in $\PP^2$.
As the points $p_i$ are general, there exists a unique curve
\be\label{j-prime}
J'\in |-K_{S'}|
\ee
 which corresponds to a smooth plane cubic passing through the $p_i$'s. We next consider the linear system on $S'$
$$
L_g:=\bigl|3g\ell-gE_1-\cdots-gE_{8}-(g-1)E_{9}\bigr|.
$$
This is a $g$-dimensional system whose general element is a smooth genus $g$ curve. Since for each curve $C'\in L_g$, we have that $C'\cdot J'=1$,  the point $\{p\}:=C'\cap J'$
is independent of $C'$ and is thus a base point of the linear system $L_g$. Precisely, $p\in J'$ is determined by the equation
$\mathcal{O}_{J'}\bigl(gp_1+\ldots+gp_8+(g-1)p_9+p\bigr)=\mathcal{O}_{J'}(3g\ell_{|J'})$.

\vskip 3pt

Let $\sigma: S\longrightarrow S'$ be the blow-up of $S'$ at $p$, We denote again by $E_1,\dots, E_{9}$ the inverse images of the exceptional curves on $S'$
and by $E_{10}$ the exceptional curve of  $\sigma$. We let $J$ be the strict transform of $J'$ and $C$ the strict transform of $C'$, so that we can write

\be\label{dv}
\aligned
-&K_S\sim J\sim 3\ell-E_1-\cdots-E_{10},\\
&C\sim 3g \ell-gE_1-\cdots-gE_{8}-(g-1)E_{9}-E_{10}\,,\\
&C\cdot J=0\,.
\endaligned
\ee

The linear system $|C|$  is base-point-free and maps $S$ to a surface $\ov S\subset \PP^g$ having canonical sections and a single elliptic
singularity  resulting from the contraction of $J$. As we mentioned above,  this linear system was first studied by Du Val in \cite{DuVal}.
\begin{defin}\label{du-val-curve} A curve in the linear system $|C|$ as in (\ref{dv}) is called a Du Val curve.
\end{defin}
In \cite{ABS2} it is proved that  Brill-Noether-Petri general curves whose Wahl map $$\nu:\bigwedge ^2 H^0(C,\omega_C)\rightarrow H^0(C,\omega_C^{\otimes 3})$$ is not surjective, are hyperplane sections of a $K3$ surface, or limits of such, and  it is shown that one such limit could be the surface $\ov S$ we just described.  This is one of the reasons why it is interesting to determine whether Du Val curves are Brill-Noether-Petri general. In this note we answer this question in the affirmative.

\begin{theorem}\label{duval1}
A general Du Val curve $C\subset S$ satisfies the Brill-Noether-Petri Theorem.
\end{theorem}
This, on the one hand, gives a strong indication that the result in \cite{ABS2} is the best possible. On the other hand, and more importantly, Theorem \ref{duval1} provides a very concrete example of a Brill-Noether-Petri curve for every value of the genus. Since the locus of $3g$-general sets of $9$ points is Zariski open in the symmetric product $(\PP^2)^{(9)}$, we can choose $p_1, \ldots, p_9$ to have rational coefficients. Then Theorem \ref{duval1} implies the following result, which answers a question raised by Harris-Morrison in \cite{HMo} p.343, in connection with the Lang-Mordell Conjecture:

\begin{cor}\label{rat}
For every $g$, there exist smooth Brill-Noether-Petri general curves $C$ of genus $g$ defined over $\mathbb Q$.
\end{cor}

In Section 5 we make Theorem \ref{duval1} and Corollary \ref{rat} explicit by proving that the following set of $9$ points in $\mathbb A^2(\mathbb Q)\subset \mathbb P^2(\mathbb Q)$,  lying on the elliptic curve $y^2=x^3+17$, is $3g$-general, for every $g$, in particular they can be used to construct Petri general curves of any genus:
$$(-2,3),(-1,-4),(2,5),(4,9),(52,375), (5234, 37866),(8, -23), (43, 282), \Bigl(\frac{1}{4}, -\frac{33}{8} \Bigr)$$

\vskip 3pt

We give two proofs of Theorem \ref{duval1}. The first one, in Section 3, uses  \cite{Nagata2} and holds for every $3g$-general set of points $p_1, \ldots, p_9$ in $\PP^2$.
The second proof, presented in Section 4, works only for a general sets of points $p_1, \ldots, p_9$,   and relies on the theory of limit linear series
and the proof of the Gieseker-Petri theorem in \cite{eis-harris}.
\vskip 0.2 cm

\subsection*{Acknowledgements} We are especially grateful to Bjorn Poonen for his help regarding Section 5 and to Edoardo Sernesi for numerous conversations  on the topic of surfaces with canonical sections. We thank Daniele Agostini, Rob Lazarsfeld and Frank-Olaf Schreyer for interesting discussions related to this circle of ideas. The last named author thanks the participants of the Oberwolfach workshop on {\it Singular curves on K3 surfaces and hyperk\"ahler manifolds} for interesting conversations on the same topic. Finally, we thank the referee for very helpful comments on the original version of this paper.

\section{Preliminaries}
As in the introduction, we denote by $S'$ the blow-up of $ \PP^2$ at nine points
 $p_1,\dots, p_{9}$ and let $E_1,\dots, E_{9}$ be the corresponding exceptional curves on $S'$.
We then consider the anticanonical elliptic curve $J'\subset S'$ as in (\ref{j-prime}).

\begin{defin}\label{cremona}
The points $p_1, \ldots, p_9$ are said to be $k$-\emph{Cremona general} for a positive integer $k$, if there exists a single cubic curve passing through them and the surface $S'$ carries no effective $(-2)$-curve of degree at most $k$. The points $p_1, \ldots, p_9$ are
\emph{Cremona general}, if they are $k$-Cremona general for any $k>0$.
\end{defin}

Nagata \cite{Nagata2} has obtained an explicit characterization of the sets of Cremona special sets, which we now explain. A permutation $\sigma \in \mathfrak{S}_9$ gives rise to an isomorphism $\sigma:\mbox{Pic}(S')\rightarrow \mbox{Pic}(S')$ induced by permuting the curves $E_1, \ldots, E_9$. We define the following divisors on $S'$:
$$\mathfrak{A}_1:=\ell-E_1-E_2-E_3, \ \ \ \mathfrak{A}_2:=2\ell-E_1-\cdots-E_6, $$
$$\mathfrak{A}_3:=3\ell-2E_1- E_2-\cdots -E_8 \  \mbox{ and } \ \mathfrak{B}=3\ell-\sum_{i=1}^9 E_i.$$

It is shown in \cite{Nagata2} Proposition 9 and Proposition 10, that a set $p_1, \ldots, p_9$ consisting of distinct points is $k$-Cremona general  if and only if the following conditions are satisfied for all permutations $\sigma\in \mathfrak{S}_9$:
\be\label{cremgenl}
\bigl|\sigma(n\mathfrak{B}+\mathfrak{A_i})\bigr|=\emptyset, \ \mbox{ for all } n\leq \frac{k-i}{3} \ \mbox{ and } \ i=1,2,3.
\ee
Since the virtual dimension of each linear system $\bigl|n\mathfrak{B}+\mathfrak{A}_i\bigr|$ is negative, clearly a very general set of points $p_1, \ldots, p_9$ is Cremona general.

\vskip 4pt

We now recall the following classical definition:
\begin{defin}\label{halph}
The points $p_1, \ldots, p_9$ are said to be $k$-\emph{Halphen special} if
there exists a plane curve of degree $3d\leq k$ having points of multiplicity $d$ at $p_1, \ldots, p_9$ and no further singularities.
We say that the set $p_1, \ldots, p_9$ is $k$-\emph{general} if it is simultaneously $k$-Cremona and $k$-Halphen general.
\end{defin}

The locus of $k$-special points defines a proper Zariski closed subvariety of the symmetric product $(\PP^2)^{(9)}$.
If $p_1, \ldots, p_9$ is a $k$-Halphen special set, then $\mbox{dim } |d J'|=1$, thus $S'\rightarrow \PP^1$ is an elliptic surface with a fibre of multiplicity $d\leq \frac{k}{3}$.
If $\mbox{Halph}(k)\subset (\PP^2)^{(9)}$ denotes the locus of $k$-special Halphen sets, then the quotient $\mbox{Halph}(k)\dblq SL(3)$ is a variety of dimension $9$, see \cite{CD} Remark 2.8.

\vskip 3pt

The relevance of both Definitions \ref{cremona} and \ref{halph} comes to the fore in the following result, which is essentially due to Nagata \cite{Nagata2}, see also \cite{DeFernex10}.

\begin{prop}\label{very-general} The  points  $p_1,\dots, p_{9}$ are $k$-\emph{general} if and only if, for every effective divisor $D$ on $S'$ such that
\be\label{div-very-gen}
D\sim d\ell-\sum_{i=1}^9\nu_i E_i\,,\quad\nu_i\geq0\,,\quad\quad \text{and}\quad D\cdot J'=0,
\ee
where $d\leq k$, one  has $D=mJ'$, for some $m$.
\end{prop}
\begin{proof}
Clearly we may assume  that $D$ is  irreducible. From the Hodge Index Theorem, it follows that $D^2\leq 0$. If $D^2<0$, then by adjunction $D$ is a smooth rational curve with $D^2=-2$. But $S'$ has no $(-2)$-curves of degree at most $k$, for $p_1, \ldots, p_9$ are $k$-Cremona general. If $D^2=0$, then applying again the Hodge Index Theorem we obtain that $D^{\perp}=K_{S'}^{\perp}$, therefore $D\in |J'|$. Thus, for an arbitrary effective divisor $D$, with $D\cdot J'=0$, we get that
$D\in |mJ'|$, for some positive integer $m\leq \frac{k}{3}$. From the $k$-Halphen generality condition, we obtain $\mbox{dim } |mJ'|=0$, hence $D=mJ'$.
The reverse implication follows directly from the definition of a $k$-general nine-tuple of points.
\end{proof}

Recall Definition \ref{du-val-curve}.
\begin{lem}\label{duval-smooth}
If the points $p_1, \ldots, p_9$ are $3$-general, a general Du Val curve  of genus $g$ is smooth and
irreducible.
\end{lem}
\begin{proof}
The linear system $|C|$ on $S$ satisfies the hypothesis of Theorem 3.1 in \cite{Harbourne} and it is then free of fixed divisors.
In particular, since by hypothesis $J$ is fixed, the general element of $|C|$ does not contain $J$. From Corollary 3.4 of  \cite{Harbourne} the
linear system $|C|$ is also base point free. This property together with Bertini's theorem and the fact that $C^2>0$, implies that the general element of $C$ is irreducible and hence smooth.

\end{proof}

\section {A general Du Val curve is a Petri general curve.}

Let $|C|$ and $S$ be as in the Introduction.  By Lemma \ref{duval-smooth}, a general element $C$ of
the  linear system $|C|$ is smooth.
Let $L$ be a base-point-free line bundle on $C$ with $h^0(C,L)=r+1$
and consider the homomorphism $\mu_{0, L}$ given by multiplication of global sections
$$
\mu_{0, L}: H^0(C, L)\otimes H^0(C, \omega_C\otimes L^{-1})\lra H^0(C, \omega_C)
$$
The curve $C$ is said to be a \emph{Brill-Noether-Petri general} curve, if the
map $\mu_{0, L}$ is injective for every line bundle $L$ on $C$. Consider the Lazarsfeld-Mukai bundle defined by the sequence
$$
0 \lra F_L\lra H^0(C, L)\otimes\mcO_S\lra L\lra0.
$$
Note that $H^0(S, F_L)=0$ and $H^1(S, F_L)=0$. Setting, as usual, $E_L:=F_L^{\vee}$, dually, we obtain the exact sequence
\begin{equation}\label{exseq2}
0 \lra H^0(C, L)^\vee\otimes\mcO_S\lra E_L\lra \omega_C\otimes L^{-1}\lra0.
\end{equation}
Here we have used that ${\omega_S}_{|C}=\mathcal O_C$. Clearly $c_1(E_L)=\mathcal{O}_S(C)$, but unlike in the $K3$ situation, on $S$
we have that $H^1(S,E_L)\cong H^0(C,L)^{\vee}$ is $(r+1)$-dimensional (rather than trivial). Following closely Pareschi's proof of  Lazarsfeld's Theorem,   \cite{Pareschi}, \cite{Lazarsfeld}, (see also Chapter XXI, section 7 of \cite{GACII}),
one proves the following lemma.
\begin{lem}If $h^0(S, F_L^\vee\otimes F_L)=1$, then $\ker \ \mu_{0, L}=0$.
\end{lem}
\begin{proof}
For the benefit of the reader we outline the proof of this Lemma following very closely the treatment in \cite{GACII}.  By tensoring the exact sequence (\ref{exseq2}) by $F_L$ and taking cohomology, since $H^0(S,F_L)=0$ and $H^1(S,F_L)=0$, we obtain
$$
H^0(S, F_L^{\vee}\otimes F_L)\cong H^0\bigl(C, F_{L|C}\otimes \omega_C\otimes L^{-1}\bigr).
$$
The twist by $\omega_C\otimes L^{-1}$ of the restriction $F_{L|C}$ of the Lazarsfeld-Mukai bundle to $C$ sits in an exact sequence
\be
\label{exmu0}
0\lra \mathcal{O}_C\lra F_{L|C}\otimes \omega_C\otimes L^{-1}\lra M_L\otimes \omega_C\otimes L^{-1}\lra 0
\ee
Moreover there is a canonical isomorphism $\ker \ \mu_{0, L}\cong H^0(C, M_L\otimes \omega_C\otimes L^{-1})$. Proposition 5.29 and diagrams (6.1) and (6.2) in \cite{GACII} show that if
$\eta:\W^r_d\to M_g$ is the family of $|L|=g^r_d$'s over moduli, then
 the image of $d\eta$ at a point $[C, L]$, is contained in
$$
(\operatorname{Im}\mu_1)^\perp\subset H^1(C, T_C)
$$
where
$$
\mu_1: \ker \mu_0 \longrightarrow H^0(C, K_C^2)=H^1(C, T_C)^\vee
$$
is the Gaussian map defined by diagram (6.1) in \cite{GACII}. We must show that the  coboundary map $\delta$ of the cohomology sequence (\ref{exmu0}) vanishes.
Let $U \subset |C|$ be the open subscheme parametrising
smooth Du Val curves in the linear system $|C|$ on $S$, and let $f:\C\to U\subset |C|$ be the family of smooth curves parametrised by $U$.
Since $S$ is regular, the characteristic map induces an isomorphism $T_{[C]}(U) \cong H^0(C, N_{C/S})$.
Consider the relative family $$p: \W^r_d(f)\to U$$
Since $p$ is surjective and $C$ is a general element of $U$, the differential
$$
dp: T_{[C,L]}(\W^r_d(f))\longrightarrow T_{[C]}(U)=H^0(C, N_{C/S})
$$
is surjective. Since $K_{S|C}=\mathcal{O}_C$, we have $N_{C/S}=\omega_C$. Let $\rho: H^0(C, N_{C/S})\to H^1(C, T_C)$
be  the Kodaira-Spencer map of the family $f$.
We then get
$$
\operatorname{Im}(\rho\circ dp)\subset \operatorname{Im}(\rho)\cap \operatorname{Im}(\mu_1)^\perp,
$$
hence
$$
\operatorname{Im}( dp)\subset \operatorname{Im}(\rho^\vee\circ \mu_1)^\perp\subset H^0(C, N_{C/S}).
$$
We set
$$
\mu_{1,S}:=\rho^\vee\circ \mu_1 :\ker \mu_0\to H^1(C, \omega_C\otimes  N^\vee_{C/S}).
$$
Since $dp$ is surjective, we get
$$
\mu_{1,S}=0
$$
In Lemma (7.9) of  \cite{GACII}, using only the fact that $K_{S|C}$ is trivial on $C$,  it is proved that $\mu_{1,S}=\delta$
up to multiplication by a nonzero scalar.  Hence  the coboundary map $\delta$ is zero. \end{proof}

\vskip 5pt

Let us go back to the construction of $S$ and $S'$, and recall the role played by the points $p_1,\dots, p_9$.
From the Riemann-Roch theorem on $S'$, these points are $3g$-Halphen general if and only if
\be
\label{bn}
H^0(J',\mcO_{J'}(kJ'))=H^0\Bigl(J, \mcO_J(k(3\ell-E_1-\cdots-E_{9}))\Bigr)=0\,,\quad k=1,\dots, g\,\qquad (J\cong J')
\ee

\vskip 3pt

\begin{theorem}\label{duval11}
 If $p_1, \ldots, p_9$ is a  $3g$-general set, then the general element of  $|C|$ is a Brill-Noether-Petri general curve.
\end{theorem}
\begin{proof}

We use the Lemma above. By contradiction, suppose there  is a non-trivial endomorphism
$\phi\in \End(F_L^\vee,F_L^\vee)$. As in Lazarsfeld's proof, we may assume that
$\phi$ is not of maximal rank.
Consider the blow-down $\sigma:  S\to S'$. We have
$$
\sigma(E_{10})=p \,,\qquad \sigma:C\cong\sigma(C)=C' \,,\qquad  \sigma:J\cong\sigma(J)= J'
$$
Notice that
$$
 (J')^2=0\,,\qquad  J'\cdot  C'=1
$$
Let $U:=S\setminus E_{10}\cong  S'\setminus \{p\}=:V$.
Let $F$ be the sheaf defined on $S'$ by the exact sequence
$$
0 \lra F\lra H^0(C, L)\otimes\mcO_{S'}\lra  L\lra 0.
$$
Since
$$
0 \lra H^0(C, L)^\vee\otimes\mcO_{ S'}\lra  F^\vee\lra \omega_C\otimes L^{-1}(p)\lra0
$$
is exact, and $L$ is special,
$F^\vee$ is generated by global sections away from a finite set of points. Consider the restriction
$$
\phi:  F^\vee_{|V}= F^\vee_{L|U}\lra  F^\vee_{L|U}=  F^\vee_{|V}
$$
By Hartogs' Theorem, $\phi$ extends uniquely
to a homomorphism
$$
\phi': F^\vee\lra   F^\vee,
$$
which is non trivial and not of maximal rank. Let
$$
E:=\im \phi'\,,\qquad  G:=\coker \phi'\,,\qquad \ov G:=G/\T(G)\,,\qquad
$$
Set
$$
A=c_1(E)\,,\qquad B=c_1(\ov G)\,,\qquad T=c_1(\T( G))\,,\qquad
$$
therefore
$$
[C']= A+B+ T.
$$
Let us prove that $A$, $B$, and $T$ are effective or trivial. The assertion for $T$ is clear.
As for $A$ and $B$ it suffices to notice that $E$ and $\ov G$ are generated by global sections away from a finite
set of points because they are positive rank torsion free quotients of $F^\vee$ .

Since $(J')^2=0$, we have that
$$
J'\cdot  A\geq 0\,,\qquad  J'\cdot  B\geq 0\,,\qquad J'\cdot T\geq 0.
$$
Since $ C'\cdot J'=1$, either $J'\cdot A=0$ or $ J'\cdot  B=0$.
By Proposition \ref{very-general},  either
$$A=kJ' \ \mbox{ or } B =hJ',$$ with $k, h\geq 0$. Both cases lead to a contradiction.
Suppose $A=kJ'$. This means that $\mcO_{J'}(A)$ is a degree-zero line bundle.
Let us show that it is the trivial bundle. Since $E$  is globally generated
away from a finite
set of points, the same holds for the restriction of its determinant  to $J'$. Thus $h^0(J', \mcO_{J'}(A))=h^0( J', \mcO_{ J'}(k J'))\neq 0$,
which contradicts condition (\ref{bn}). To summarize, the non-trivial endomorphism $\phi$ cannot exist in the first place and $C$ is a Brill-Noether-Petri general curve.
\end{proof}
\begin{rem} If the set $p_1, \ldots, p_9$ is $3d$-Halphen special,
the linear system $|3d\ell-d\sum_{i=1}^9E_i|$  cuts out on $C$ a $g^1_d$. In particular, one can realize curves of arbitrary gonality as special Du Val curves.
\end{rem}

\section{Lefschetz pencils of Du Val curves}

In this section we determine the intersection numbers of a rational curve $j:\mathbb P^1\rightarrow \overline{\mathcal{M}}_g$ induced by a pencil of Du Val curves on $S$
with the generators of the Picard
group of the moduli space $\overline{\mathcal{M}}_g$. Recall that $\lambda$ denotes the Hodge class and $\delta_0, \ldots, \delta_{\lfloor \frac{g}{2}\rfloor}\in \mbox{Pic}(\overline{\mathcal{M}}_g)$ are the classes corresponding to the boundary divisors of the moduli space. We denote by $\delta:=\delta_0+\cdots+\delta_{\lfloor \frac{g}{2}\rfloor}$ the total boundary. For integers
$r,d\geq 1$, we denote by $\mathcal{M}_{g,d}^r$ the locus of curves $[C]\in \mathcal{M}_g$ such that $W^r_d(C)\neq \emptyset$.
If $\rho(g,r,d)=-1$, in particular $g+1$ must be composite, $\mathcal{M}_{g,d}^r$ is an effective divisor.
Eisenbud and Harris \cite{EH} famously computed the class of the closure of the Brill-Noether divisors:
\begin{equation}\label{bndivs}
[\overline{\mathcal{M}}_{g,d}^r]=c_{g.d,r}\Bigl((g+3)\lambda-\frac{g+1}{6}\delta_0-\sum_{i=1}^{\lfloor \frac{g}{2}\rfloor} i(g-i)\delta_i\Bigr)\in \mathrm{Pic}(\overline{\mathcal{M}}_g).
\end{equation}

\vskip 3pt

We retain the notation of the introduction and observe that the linear system
$$\Lambda_{g-1}:=\bigl|3(g-1)\ell-(g-1)E_1-\cdots-(g-1)E_8-(g-2)E_9\bigr|$$
appears as a hyperplane in the $g$-dimensional linear system $L_g$ on the surface $S$. It consists precisely of the
curves $D+J\in L_g$, where $D\in \Lambda_{g-1}$. We now choose a Lefschetz pencil in $L_g$, which has $2g-2=C^2$ base points.
Let $X:=\mbox{Bl}_{2g-2}(S)$ be the blow-up of $S$ at those points and we denote by $f:X\rightarrow \mathbb P^1$ the induced fibration, which gives rise to
a moduli map
$$j:\mathbb P^1\rightarrow \overline{\mathcal{M}}_g.$$

We compute the numerical features of this Du Val pencil in the moduli space:

\begin{theorem}\label{pencil}
The intersection numbers of the Du Val pencil with the generators of the Picard group of $\overline{\mathcal{M}}_g$ are given as follows:
$$
j^*(\lambda)=g, \ \ j^*(\delta_0)=6(g+1), \ \ j^*(\delta_1)=1, \ \ \mbox{ and } j^*(\delta_i)=0 \ \mbox{ for } i=2, \ldots, \bigl\lfloor \frac{g}{2}\bigr\rfloor.
$$
As a consequence: $j^*([\overline{\mathcal{M}}_{g,d}^r])=0$.
\end{theorem}
\begin{proof}
Using Grothendieck-Riemann-Roch, we have the following formulas valid for the moduli map $j$ induced by $f:X\rightarrow \mathbb P^1$:
$$
j^*(\lambda)=\chi(X,\mathcal{O}_X)+g-1, \ \ \ j^*(\delta)=c_2(X)+4(g-1).
$$
Clearly $\chi(X,\mathcal{O}_X)=1$, therefore $j^*(\lambda)=g$. Furthermore, since $X$ is $\PP^2$ blown up at $2g+8$ points, $c_2(X)=12\chi(X, \mathcal{O}_X)-K_X^2=2g+11$,
and accordingly $j^*(\delta)=6g+7$. Of these $6g+7$ singular curves in the pencil, there is precisely one of type $D+J$, where $D$ is the proper transform
of a curve in the linear system $\Lambda_{g-1}$. Note that $D\cdot J=1$. Therefore $j^*(\delta_1)=1$. A parameter count also shows that a general Du Val
pencil contains no curves in the higher boundary divisors $\Delta_i$, where $i\geq 2$, therefore $j^*(\delta_0)=6g+6$.
Using (\ref{bndivs}), we now compute $j^*([\overline{\mathcal{M}}_{g,d}^r])=0$, and finish the proof.
\end{proof}

We record the following immediate consequence of Theorem \ref{pencil}
\begin{cor}\label{lefschetz}
For any choice of nine distinct points $p_1, \ldots, p_9\in \PP^2$, the Du Val pencil $j(\mathbb P^1)$ either lies entirely in or is disjoint from any Brill-Noether divisor
$\overline{\mathcal{M}}_{g,d}^r$.
\end{cor}

In particular, notice that when the points $p_1, \ldots, p_9$ belong to the Halphen stratum $\mbox{Halp}(3d)$, then the elliptic pencil $|dJ'|$ on $S'$ cut out a pencil of degree $d$ on each curve $C'$, in particular $\mbox{gon}(C)\leq d$. Such Halphen surfaces $S$, appear as limits of polarised $K3$ surfaces $(X,H)$, where $X$ carries an elliptic pencil $|E|$ with $E\cdot H=k$. The enlargement of the Picard group on the side of $K3$ surfaces correspond on the Du Val side to the points $p_1, \ldots, p_9$ becoming Halphen special.

\vskip 5pt
\begin{rem} Du Val curves of genus $g$ form a unirational subvariety  of dimension $$\mathrm{min}(g+10, 3g-3)$$ inside the moduli space $\mathcal{M}_g$. In particular, for $g=7$, one has a divisor $\mathfrak{Dv}_7$ of Du Val curves of genus $7$. It would be interesting to describe this divisor and compute the class $[\overline{\mathfrak{Dv}}_7]\in \mathrm{Pic}(\overline{\mathcal{M}}_7)$.
\end{rem}

\subsection{Du Val curves are Petri general: a second proof.} We now describe an alternative approach, based on the theory of limit linear series, to prove a slightly weaker version of Theorem  \ref{duval1}. We retain throughout the notation of the Introduction. We denote by $\mathcal{BN}$ (respectively $\mathcal{GP}$) the proper subvariety of $\mathcal{M}_g$ consisting of curves $[C]$ having a line bundle $L$ which violates the Brill-Noether (respectively the Gieseker-Petri) condition. Clearly $\mathcal{BN}\subset \mathcal{GP}$.

\begin{theorem}\label{proof2}
Let $S'$ be the blow-up of $\PP^2$ at nine \emph{general} points $p_1, \ldots, p_9$ and set as before $$L_g:=\bigl| 3g\ell-gE_1-\cdots -gE_8-(g-1)E_9\bigr|.$$ Then a general curve $C'\in L_g$ satisfies the Petri Theorem. Furthermore, an arbitrary irreducible nodal curve $C'\in L_g$ satisfies the Brill-Noether Theorem.
\end{theorem}

\begin{proof}
Assume by contradiction, that for a general choice of $p_1, \ldots, p_9\in \PP^2$, there exists a nodal curve $C'\in L_g$ that violates the Brill-Noether condition. We let the points
$p_1, \ldots, p_9$ specialize to the base locus of a general pencil of plane cubics. Then $S'$ becomes a rational elliptic surface $\pi:S'\rightarrow \PP^1$ and $E:=E_9$ can be viewed as a section of $\pi$.

By a standard calculation, since $\pi_* \mathcal{O}_{S'}=\mathcal{O}_{\PP^1}$, we compute that
$$h^0(S', \mathcal{O}_{S'}(gJ'))=h^0(\PP^1, \mathcal{O}_{\PP^1}(g))=g+1.$$ Similarly, since $\pi_*(\mathcal{O}_{S'}(E))=\mathcal{O}_{\PP^1}$, we find that $h^0(S', \mathcal{O}_{S'}(gJ'+E))=g+1$. Therefore, every element of the linear system $L_g$ is of the form
$J_1+\cdots+J_g+E$, where $J_i\in |\mathcal{O}_{S'}(J')|$ are elliptic curves on $S'$ and $J_i\cdot E=1$, for $i=1, \ldots, g$.

\vskip 3pt

Let $\varphi:\overline{\mathcal{M}}_{0,g}\times \overline{\mathcal{M}}_{1,1}^g\rightarrow \overline{\mathcal{M}}_g$ be the map obtained by attaching to each $g$-pointed stable rational curve $[R, x_1, \ldots, x_g]\in \overline{\mathcal{M}}_{0,g}$ elliptic tails $J_1, \ldots, J_g$ at the points $x_1, \ldots, x_g$ respectively. The symmetric group $\mathfrak{S}_g$ acts diagonally on the product $\overline{\mathcal{M}}_{0,g}\times \overline{\mathcal{M}}_{1,1}^g$, by simultaneously permuting the markings $x_i$ and the tails $J_i$ for $i=1, \ldots, g$. The map $\varphi$ is $\mathfrak{S}_g$-invariant. Observe that the moduli map  $m:L_g\dashrightarrow \overline{\mathcal{M}}_g$ corresponding to the linear system $L_g$ factors via $\bigl(\overline{\mathcal{M}}_{0,g}\times \overline{\mathcal{M}}_{1,1}^g\bigr)/\mathfrak{S}_g$. Since the morphism $\varphi$ is regular, it follows that the variety of stable limits of $L_g$, defined as the image $\pi_2(\Sigma)$ of the graph $L_g\stackrel{\pi_1}\longleftarrow \Sigma\stackrel{\pi_2}\longrightarrow \overline{\mathcal{M}}_g$ of the rational map $m$, is actually contained in $\mbox{Im}(\varphi)$.

\vskip 3pt

Using \cite{EH} Theorem 1.1, no curve lying $\mbox{Im}(\varphi)$ carries a limit linear series $\mathfrak g^r_d$ with negative Brill-Noether number (note that all the stable curves in $\mbox{Im}(\varphi)$ are \emph{tree-like} in the sense of \cite{EH}, so the theory of limit linear series applies to them). It follows that $\mbox{Im}(\varphi) \cap \overline{\mathfrak{BN}}=\emptyset$.

\vskip 3pt

Our hypothesis implies that we can find a family of Du Val curve $f:\mathcal{C}\rightarrow (T,0)$ over a $1$-dimensional base, such that for the general fibre $[f^{-1}(t)]\in \mathcal{BN}$, whereas the central fibre $f^{-1}(0)$ is a (possibly non-reduced) curve from the linear system $L_g$. Applying stable reduction to $f$, we obtain a new family having in the central fibre a stable curve that lies simultaneously in $\mbox{Im}(\varphi)$ and in $\overline{\mathcal{BN}}$, which is a contradiction.

Furthermore, the proof of the Gieseker-Petri Theorem in \cite{eis-harris}, implies that for any choice of elliptic tails $[J_1,x_1] \ldots, [J_g,x_g]\in \overline{\mathcal{M}}_{1,1}^g$, there exists $[R, x_1, \ldots, x_g]\in \overline{\mathcal{M}}_{0,g}$ such that $\varphi\Bigl([R, x_1, \ldots, x_g], [J_1, x_1], \ldots, [J_g, x_g]\Bigr)\not\in \overline{\mathcal{GP}}$. This implies that for general $p_1, \ldots, p_9\in \PP^2$, a general curve $C'\in L_g$ satisfies Petri's condition.
\end{proof}


\begin{rem} The conclusion of Theorems \ref{duval1} and \ref{proof2} cannot be improved, in the sense that it is \emph{not true} that every smooth curve $C'\in L_g$ verifies the Petri condition. The classes of the closure of the divisorial components $\mathcal{GP}_{g,d}^r$ of $\mathcal{GP}$ corresponding to line bundles $L\in W^r_d(C)$ such that $g-(r+1)(g-d+r)=0$, have been computed in \cite{EH} Section 5, when $r=1$ and in \cite{Farkas} Theorem 1.6 in general. Taking the pencil $j:\PP^1\rightarrow \overline{\mathcal{M}}_g$ considered in Theorem \ref{pencil}, we immediately conclude that $j^*([\overline{\mathcal{GP}}_{g,d}^r])\neq 0$.
\end{rem}

\section{An explicit system of nine general points}

In this final section we show how, using standard techniques from the arithmetic of elliptic curves, we can exhibit an explicit system of nine points verifying the genericity assumption of Definition \ref{halph} for every $k$.
Throughout this section we use the embedding $\mathbb A^2(\mathbb Q)\hookrightarrow \mathbb P^2(\mathbb Q)$.

We start with the elliptic curve $E: y^2=x^3+17$, and we denote by $p_{\infty}:=[0,1,0]\in E$ its point at infinity and use the identification $\mathcal{O}_E(1)=\mathcal{O}_E(3p_{\infty})$. If $q\in E$, we denote by $-q\in E$ its inverse element using the group law of $E$, having $p_{\infty}$ as origin.
Observe that the following points belong to $E(\mathbb Q)$:
$$p_1=(-2,3), \ p_2=(-1,-4),\ p_3=(2,5), \ p_4=(4,9), \ p_5=(52,375),$$
as well as,
$$ p_6=(5234, 37866), \ p_7=(8, -23),  \ p_8=(43, 282), \ \mbox{ and } p_9=\Bigl(\frac{1}{4}, -\frac{33}{8} \Bigr).$$
It is known that $\pm{p_i}$ for $i=1, \ldots, 8$ are the only points in $E(\mathbb Z)-\{0\}$.
Using the explicit formulas for the addition law on $E$, observe that $p_4=p_1-p_3$, $p_2=2p_1-p_3$,  $p_5=3p_1-p_3$,  $p_6=4p_1-3p_3$, \ $p_7=2p_1$,\ $p_8=2p_3-p_1$ and $p_9=p_1+p_3$.
The following facts are known to experts, we include an elementary proof for the sake of completeness.

\begin{lem}\label{points9}
1) One has $E(\mathbb Q)_{\mathrm{tors}}=0$.
\newline\noindent
2) One has an embedding $\mathbb Z\cdot p_1\oplus \mathbb Z\cdot p_3\hookrightarrow E(\mathbb Q)$\footnote{In fact one can prove that $E(\mathbb Q)=\mathbb Z\oplus \mathbb Z$, that is, each rational point of $E$ can be written as a unique combination of $p_1$ and $p_3$, see \cite{Nag} or use the program PARI, but we will not use this fact.}.
\end{lem}
\begin{proof}
For the first part, we use that if $p$ is a prime not dividing the discriminant of $E$, one has an embedding $E(\mathbb Q)_{\mathrm{tors}}\hookrightarrow E(\mathbb F_p)$, see for instance \cite{Silv} Chapter 7. The curve $E$ has good reduction at the primes $5$ and $7$
(in fact, at any prime different from $2, 3$ and $17$). Therefore, the torsion subgroup $E(\mathbb Q)_{\mathrm{tors}}$ injects into both $E(\mathbb F_5)$ and $E(\mathbb F_7)$, which are of orders $6$ and $13$, respectively. It follows that $E(\mathbb Q)_{\mathrm{tors}}$ is trivial. We remark that the same conclusion can be obtained by applying the Nagell-Lutz Theorem.

\vskip 3pt

We prove that the points $p_1$ and $p_3$ are independent in $E(\mathbb Q)$. Since $E(\mathbb Q)[2]=0$, it will suffice to show that no linear combination $np_1+mp_3$ of the points $p_1=(-2,3)$, $p_3=(2,5)$ can be zero, where at least one of $m, n\in \mathbb Z$ is odd. This follows once we show that $p_1, p_3$, as well as  $p_4=p_1-p_3=(4,9)$ are non-zero in the quotient $E(\mathbb Z)/2E(\mathbb Z)$. Recall \cite{Silv} page 58, that if $p=(a,b)\in E(\mathbb Q)$, then the $x$-coordinate of the point $2p\in E$ is given by
$$x(2p)=\frac{a^4-136}{4a^3+68}.$$
Assuming $p_1\in 2E({\mathbb Z})$, we obtain that the equation $a^4-136a=8(a^3+17)$ has an integral solution, which is a contradiction. The proof that $p_3\not\in 2E(\mathbb Z)$ is identical. If $p_4\in 2E(\mathbb Z)$, then the equation $a^4-136a=16(a^3+17)$ has an integral solution, again a contradiction.
\end{proof}

\begin{theorem} The points $p_1, \ldots, p_9$ are $k$-general for every integer $k$.
\end{theorem}
\begin{proof}
The condition that the nine points are $k$-Halphen special for some $k\geq 0$ is precisely that $p_1+\cdots+p_9\in E(\mathbb Q)_{\mathrm{tors}}$, that is, $p_1+\cdots+p_9=13p_1-p_3=0$, which contradicts  Lemma \ref{points9}.

To show that the points are Cremona general, we unwind the conditions appearing in  (\ref{cremgenl}) in terms of the group law on $E$. In turns out that if $p_1, \ldots, p_9$ are Cremona general, then there exists non-negative integers $n_1, \ldots, n_9$, not all equal to zero, such that the linear equivalence $n_1p_1+\cdots+n_9p_9\equiv (n_1+\cdots+n_9)p_{\infty}$ holds, that is, $n_1p_1+\cdots+n_9p_9=0\in E$. Since with the exception of $p_4=p_1-p_3$, each of the points $p_1, \ldots, p_9$ are combinations of the type $mp_1+np_3$, with $m+n>0$, we obtain that such a combination of $p_1$ and $p_3$ is equal to zero, which contradicts the second part of Lemma \ref{points9}.
\end{proof}

\bibliographystyle{abbrv}

\bibliography{bibliografia}

\begin{thebibliography}{10}

\bibitem{andre}
Y.~Andr\'e.
\newblock Pour une th\'eorie inconditionnelle des motifs,.
\newblock {\em Publ. Math.Inst. Hautes \'Etudes Sci.}, 83:5--49, 1996.

\bibitem{ABS2}
E.~Arbarello, A.~Bruno, and E.~Sernesi.
\newblock On two conjectures by {J}.{W}ahl.
\newblock {\em arXiv:1507.05002v2}, 2015.

\bibitem{GACII}
E.~Arbarello, M.~Cornalba, and P.~A. Griffiths.
\newblock {\em Geometry of algebraic curves. {V}olume {II}}, volume 268 of {\em
  Grundlehren der Mathematischen Wissenschaften [Fundamental Principles of
  Mathematical Sciences]}.
\newblock Springer, Heidelberg, 2011.
\newblock With a contribution by Joseph Daniel Harris.

\bibitem{CD}
S.~Cantat and I.~Dolgachev.
\newblock {R}ational surfaces with a large group of automorphisms.
\newblock {\em J. American Math. Soc.}, 25:863--905, 2012.

\bibitem{Payne}
F.~Cools, J.~Draisma, S.~Payne, and E.~Robeva.
\newblock A tropical proof of the {B}rill-{N}oether {T}heorem.
\newblock {\em Advances in Math.}, 230:759--776, 2012.

\bibitem{DeFernex10}
T.~De~Fernex.
\newblock On the {M}ori cone of blow-ups of the plane.
\newblock {\em arXiv:1001.5243v2}, 2010.

\bibitem{DuVal}
P.~du~Val.
\newblock On {R}ational {S}urfaces whose {P}rime {S}ections are {C}anonical
  {C}urves.
\newblock {\em Proc. London Math. Soc.}, series 2-35:1--13, 1933.

\bibitem{eis-harris}
D.~Eisenbud and J.~Harris.
\newblock A simple proof of the {G}ieseker-{P}etri theorem on special divisors.
\newblock {\em Invent. Math.}, 74:269--280, 1983.

\bibitem{EH}
D.~Eisenbud and J.~Harris.
\newblock The {K}odaira dimension of the moduli space of curves of genus $\geq
  23$.
\newblock {\em Invent. Math.}, 90:359--387, 1987.

\bibitem{Farkas}
G.~Farkas.
\newblock {K}oszul divisors on moduli spaces of curves.
\newblock {\em American Journal of Math.}, 131:819--896, 2009.

\bibitem{Gieseker}
D.~Gieseker.
\newblock Stable curves and special divisors: {P}etri's conjecture.
\newblock {\em Invent. Math.}, 66:251--275, 1982.

\bibitem{GH}
P.~Griffiths and J.~Harris.
\newblock On the variety of special linear systems on a general algebraic
  curve.
\newblock {\em Duke Math. J.}, 47:233--272, 1980.

\bibitem{Harbourne}
B.~Harbourne.
\newblock Complete linear systems on rational surfaces.
\newblock {\em Trans. Amer. Math. Soc.}, 289(1):213--226, 1985.

\bibitem{HMo}
J.~Harris and I.~Morrison.
\newblock {\em Moduli of curves}, volume 187 of {\em Graduate Texts in
  Mathematics}.
\newblock Springer, Heidelberg, 1998.

\bibitem{Lazarsfeld}
R.~Lazarsfeld.
\newblock Brill-{N}oether-{P}etri without degenerations.
\newblock {\em J. Differential Geom.}, 23:299--307, 1986.

\bibitem{MP}
D.~Maulik and B.~Poonen.
\newblock {N}\'eron-{S}everi groups under specialization.
\newblock {\em http://www-math.mit.edu/~poonen/papers/NSjumping.pdf}.

\bibitem{Nagata2}
M.~Nagata.
\newblock On rational surfaces {II}.
\newblock {\em Memoirs of the College of Science, University of Kyoto},
  32:271--293, 1960.

\bibitem{Nag}
T.~Nagell.
\newblock Solution de quelque probl\'emes dans la th\'eorie arithm\'etique des
  cubiques planes du premier genre.
\newblock {\em Wid. Akad. Skrifter Oslo I}, 1935.

\bibitem{Pareschi}
G.~Pareschi.
\newblock A proof of {L}azarsfeld's theorem on curves on {$K3$} surfaces.
\newblock {\em J. Algebraic Geom.}, 4:195--200, 1995.

\bibitem{Silv}
J.~Silverman.
\newblock {\em The arithmetic of elliptic curves}, volume 106 of {\em Graduate
  Texts in Mathematics}.
\newblock Springer, Heidelberg, 1986.

\bibitem{Treibich}
A.~Treibich.
\newblock Rev\^etements tangentiels et condition de {B}rill-{N}oether.
\newblock {\em C. R. Acad. Sci. Paris S\'er. I Math.}, 316(8):815--817, 1993.

\bibitem{Welters}
G.~Welters.
\newblock A theorem of {G}ieseker-{P}etri type for {P}rym varieties.
\newblock {\em Ann. Scient. Ecole Normale Sup.}, 18:671--683, 1985.

\end{thebibliography}

\end{document}